\documentclass[11pt]{article}

\usepackage[margin=1.15in]{geometry}
\usepackage{amsmath,amssymb,amsthm,mathtools}
\usepackage{xcolor}
\usepackage[colorlinks=true,citecolor=blue,linkcolor=blue,urlcolor=blue]{hyperref}
\usepackage{microtype}

\newtheorem{theorem}{Theorem}[section]
\newtheorem{proposition}[theorem]{Proposition}
\newtheorem{lemma}[theorem]{Lemma}
\newtheorem{corollary}[theorem]{Corollary}

\newtheorem{problem}[theorem]{Problem}
\newtheorem{claim}[theorem]{Claim}

\newcommand{\rt}{r_{<}}
\newcommand{\chit}{\chi_{<}}
\newcommand{\eps}{\varepsilon}
\newcommand{\bbP}{\mathbb{P}}
\newcommand{\per}{\operatorname{per}}

\title{Ordered matchings versus triangles via pseudorandom triangle-free graphs%
\thanks{Email: chenwenfj@163.com, linqizhong@fzu.edu.cn,
chunlin\_you@163.com. Supported in part by the National Key R\&D
Program of China (Grant No. 2023YFA1010202), the NSFC
(No. 12571361, 12401469) and the Young Elite Scientist Sponsorship
Program by JSAST (JSTJ-2025-892).}}

\author{
Wen Chen$^1$,\quad \quad
Qizhong Lin$^1$,\quad \quad Chunlin You$^2$
\\
\small 1. Center for Discrete Mathematics, Fuzhou University,
Fujian 350108, China
\\
\small 2. School of Mathematics and Statistics,
Yancheng Teachers University, Yancheng 224002, China
}

\date{}

\begin{document}
\maketitle

\begin{abstract}
For ordered graphs $H_1,\ldots,H_t$, let
$\rt(H_1,\ldots,H_t)$ denote the least integer $N$ such that every
$t$-coloring of the edges of the naturally ordered complete graph on
$[N]$ contains an ordered copy of $H_i$ in color $i$ for some $i\in[t]$.
We prove that a uniformly random ordered matching $M$ on $n$ vertices
with interval chromatic number two asymptotically almost surely satisfies
\[
   \rt(K_3,M)
   =\Omega\left(\frac{n^{4/3}}{(\log n)^{1/3}}\right).
\]
This strengthens the lower bound $\Omega((n/\log n)^{5/4})$
of Balko and Poljak for such random matchings and improves the
general existential lower bound of Conlon, Fox, Lee and Sudakov
by a factor of $\log n$.
The proof combines pseudorandom triangle-free graphs, a coarse encoding
of order-preserving embeddings, and a permutation avoidance estimate
derived from Br\`egman's inequality.
\end{abstract}


\section{Introduction}

An \emph{ordered graph} is a graph together with a linear order on its
vertex set. We identify the vertex set of an ordered graph on $n$ vertices
with $[n]=\{1,2,\dots,n\}$. An ordered graph $G$ on $[N]$ \emph{contains} an
ordered graph $H$ on $[m]$ if there exists a map $\phi\colon[m]\to[N]$ such
that $\phi(i)<\phi(j)$ whenever $1\leq i<j\leq m$, and
$\phi(i)\phi(j)\in E(G)$ whenever $ij\in E(H)$.
For ordered graphs $H_1,\ldots,H_t$, let $\rt(H_1,\ldots,H_t)$ denote the
least integer $N$ such that every $t$-coloring of the edges of the ordered
complete graph on $[N]$ contains an ordered copy of $H_i$ in color $i$ for
some $i\in[t]$.
Throughout the paper, all matchings are assumed to have no isolated
vertices.

The study of ordered Ramsey numbers goes back to the
classical work of Erd\H{o}s and Szekeres~\cite{erdosSzekeres35} and was
developed systematically by Balko, Cibulka, Kr\'al and
Kyn\v{c}l~\cite{BCKK} and by Conlon, Fox, Lee and Sudakov~\cite{CFLS}.
A striking feature of the subject is that even ordered matchings can have
superpolynomial Ramsey numbers~\cite{BCKK,CFLS}. A basic parameter of an
ordered graph is its interval chromatic number. The \emph{interval
chromatic number} $\chit(G)$ of an ordered graph $G$ is the minimum number
of consecutive intervals that partition its vertex set and contain no
edge.

Conlon, Fox, Lee and Sudakov~\cite{CFLS} raised the natural off-diagonal
problem of estimating $\rt(K_3,M)$ for ordered matchings $M$.
They proved the following lower bound.

\begin{theorem}[Conlon, Fox, Lee and Sudakov~\cite{CFLS}]
\label{thm:cfls-matching-triangle}
There exists an absolute constant $c>0$ such that, for every positive even
integer $n$, there is an ordered matching $M$ on $n$ vertices satisfying
\[
    \rt(K_3,M)
    \geq c\left(\frac{n}{\log n}\right)^{4/3}.
\]
\end{theorem}

On the other hand, the upper bound $O(n^2/\log n)$ follows from the
estimate for $r(K_3,K_n)$ due to Ajtai, Koml\'os and
Szemer\'edi~\cite{AKS}; it holds for every ordered matching $M$ on $n$
vertices. Thus the correct order of magnitude remained unknown.
Conlon, Fox, Lee and Sudakov~\cite{CFLS} therefore asked whether there
exists an $\varepsilon>0$ such that
\[
    \rt(K_3,M)=O(n^{2-\varepsilon})
\]
for every ordered matching $M$ on $n$ vertices~\cite[Problem~6.1]{CFLS}.

Subsequent work by Rohatgi~\cite{Ro19} and by Balko and
Poljak~\cite{BalkoPoljak,BP23} addressed the problem for ordered
matchings of bounded interval chromatic number. In particular, Balko and
Poljak~\cite{BalkoPoljak} showed that the lower bound of Conlon, Fox,
Lee and Sudakov also holds for some ordered matching of interval
chromatic number three. For interval chromatic number two, they proved
the following lower bound.

\begin{theorem}[Balko and Poljak~\cite{BalkoPoljak}]
There exists a constant $c>0$ such that, for every positive even integer
$n$, a uniformly random ordered matching $M$ on $n$ vertices with
$\chit(M)=2$ satisfies, asymptotically almost surely,
\[
    \rt(K_3,M)
    \geq c\left(\frac{n}{\log n}\right)^{5/4}.
\]
\end{theorem}

They also proved that such a matching satisfies, asymptotically almost
surely,
\[
    \rt(K_3,M)=O(n^{7/4}).
\]
These results left a gap between the lower bound
$\Omega((n/\log n)^{5/4})$ and the upper bound $O(n^{7/4})$,
motivating the following problem.

\begin{problem}[Balko and Poljak, Problem~8 in~\cite{BalkoPoljak}]
\label{prob:random-matching}
What is the growth rate of $\rt(K_3,M)$ for a uniformly random ordered
matching $M$ on $n$ vertices with $\chit(M)=2$?
\end{problem}

Equivalently, write $n=2m$, take $\pi\in S_m$ uniformly at random, and set
\[
    M_\pi=\{\{i,m+\pi(i)\}:1\le i\le m\}.
\]
This is exactly the uniform distribution in
Problem~\ref{prob:random-matching}: every ordered perfect matching on
$[2m]$ with $\chit(M)=2$ has its two independent intervals equal to $[m]$
and $\{m+1,\ldots,2m\}$, and hence determines a unique permutation $\pi$.

We make progress on Problem~\ref{prob:random-matching} by proving a lower
bound that is stronger than the previous bounds, even for uniformly random
matchings of interval chromatic number two.

\begin{theorem}\label{thm:main}
There is a constant \(c>0\) such that, for every sufficiently large even
integer \(n\), a uniformly random ordered matching \(M\) on \(n\) vertices
with \(\chit(M)=2\) asymptotically almost surely satisfies
\[
   \rt(K_3,M)
      \geq c\frac{n^{4/3}}{(\log n)^{1/3}}.
\]
\end{theorem}

This strengthens the lower bound $\Omega((n/\log n)^{5/4})$ of Balko and Poljak and improves the
general existential lower bound of Conlon, Fox, Lee and Sudakov by a factor of $\log n$.

The proof uses an ordered blow-up of a pseudorandom triangle-free graph
on $B$ vertices. We encode each potential embedding by rounding its
block preimages to a fixed partition of each side of the matching into
$O(B)$ consecutive intervals. Since these preimages occur in order,
there are only $\exp(O(B))$ resulting profiles. The rounding discards
only a small proportion of the vertices, and a weighted form of the
pseudorandom density estimate ensures that each profile still determines
many forbidden cells for the random permutation defining the matching.
Br\`egman's inequality bounds the probability of avoiding these cells,
allowing a union bound over all profiles.

\medskip

The remainder of the paper is organized as follows. In
Section~\ref{sec:prelim}, we develop the pseudorandom and probabilistic
tools used in the proof. In Section~\ref{sec:main-proof}, we prove
Theorem~\ref{thm:main}.

\section{Preliminaries}\label{sec:prelim}

Throughout the paper, we omit floor and ceiling signs whenever they do not
affect the argument; all such quantities are understood to be rounded to
integers in the natural way. All logarithms are to base $2$.

For a graph $H$ and disjoint vertex sets $X,Y\subseteq V(H)$, let
$E_H(X,Y)$ denote the set of edges of $H$ with one endpoint in $X$ and the
other in $Y$, and write $e_H(X,Y)=|E_H(X,Y)|$.

The following lemma is a specialization of Theorem~4 of Guo and
Warnke~\cite{GuoWarnke} to complete host graphs.
It provides a triangle-free graph with a uniform lower bound on
the edge density between any two sufficiently large disjoint
vertex sets.

\begin{lemma}[Guo and Warnke~\cite{GuoWarnke}]
\label{lem:GWsimple}
There are absolute constants \(C,\rho>0\) such that, for every sufficiently
large integer \(B\), there exists a triangle-free graph \(F\) on vertex set
\([B]\) such that, with
\[
  q_B=\sqrt{\frac{\log B}{B}},
  \qquad
  s_B=C\sqrt{B\log B},
\]
we have, for all disjoint sets \(X,Y\subseteq[B]\) satisfying \(|X|=|Y|=s_B\),
\[
  e_F(X,Y)\ge \rho q_B |X||Y|.
\]

\end{lemma}

We use only the following immediate consequence.

\begin{corollary}\label{cor:GW}
There are absolute constants \(C,\rho>0\) such that, for every sufficiently
large integer \(B\), there is a triangle-free graph \(F\) on vertex set
\([B]\) with the following property. Put
\[
  q_B=\sqrt{\frac{\log B}{B}},
  \qquad
  s_B=C\sqrt{B\log B}.
\]
Then, for all disjoint sets \(X,Y\subseteq[B]\) with
\(|X|,|Y|\ge s_B\), 
\[
  e_F(X,Y)\ge \rho q_B|X||Y|.
\]
\end{corollary}

\begin{proof}
Let \(C\) and \(\rho\) be as in Lemma~\ref{lem:GWsimple}. Fix disjoint
\(X,Y\subseteq[B]\) with \(|X|,|Y|\ge s_B\). Choose
\(X'\in\binom{X}{s_B}\) and \(Y'\in\binom{Y}{s_B}\) independently and
uniformly at random. Since \(X'\cap Y'=\varnothing\) and
\(|X'|=|Y'|=s_B\), Lemma~\ref{lem:GWsimple} gives
\[
  e_F(X',Y')\ge \rho q_B s_B^2.
\]
Taking expectations and using linearity,
\[
  \rho q_B s_B^2
  \le \mathbb E\,e_F(X',Y')
  = e_F(X,Y)\frac{s_B}{|X|}\frac{s_B}{|Y|},
\]
because each edge between \(X\) and \(Y\) is included with probability
\(\frac{s_B}{|X|}\frac{s_B}{|Y|}\). Hence
\[
  e_F(X,Y)\ge \rho q_B|X||Y|.
\]
\end{proof}

The next lemma extends the lower bound for pairs of sets to a weighted
lower bound.

\begin{lemma}\label{lem:weighted-cut-density}
Let \(F\) be a graph, and let \(q,\rho>0\) and \(s\ge1\) be such that
\[
e_F(X,Y)\ge \rho q|X||Y|
\]
for all disjoint sets \(X,Y\subseteq V(F)\) with \(|X|,|Y|\ge s\).
Let \(P,Q\subseteq V(F)\) be disjoint and nonempty.
Let \((w_x)_{x\in P}\) and \((w_y)_{y\in Q}\) be nonnegative integer
weights, and set
\[
W_P=\sum_{x\in P}w_x,
\qquad W_Q=\sum_{y\in Q}w_y,
\qquad \Delta_P=\max_{x\in P}w_x,
\qquad \Delta_Q=\max_{y\in Q}w_y.
\]
If \(W_P\ge s\Delta_P\) and \(W_Q\ge s\Delta_Q\), then
\[
\sum_{xy\in E_F(P,Q)}w_xw_y
\ge
\rho q\,(W_P-s\Delta_P)(W_Q-s\Delta_Q).
\]
\end{lemma}

\begin{proof}
For positive integers \(r,t\), define
\[
   P_r=\{x\in P:w_x\ge r\},
   \qquad
   Q_t=\{y\in Q:w_y\ge t\}.
\]
Each edge \(xy\in E_F(P,Q)\) is counted for exactly \(w_xw_y\) pairs
\((r,t)\) with \(x\in P_r\) and \(y\in Q_t\).
Summing over the edges between \(P\) and \(Q\) therefore yields the
layer-counting identity
\begin{equation}\label{eq:weighted-layer-count}
   \sum_{xy\in E_F(P,Q)}w_xw_y
   =\sum_{r\ge1}\sum_{t\ge1}e_F(P_r,Q_t).
\end{equation}
Since \(P_r=\varnothing\) for \(r>\Delta_P\) and
\(Q_t=\varnothing\) for \(t>\Delta_Q\),
\[
   \sum_{\substack{1\le r\le \Delta_P\\ |P_r|<s}}|P_r|
   \le s\Delta_P,
   \qquad
   \sum_{\substack{1\le t\le \Delta_Q\\ |Q_t|<s}}|Q_t|
   \le s\Delta_Q.
\]
Moreover, double counting the pairs \((x,r)\) with \(x\in P_r\), and
similarly the pairs \((y,t)\) with \(y\in Q_t\), we obtain
\[
   \sum_{r=1}^{\Delta_P}|P_r|=W_P,
   \qquad
   \sum_{t=1}^{\Delta_Q}|Q_t|=W_Q.
\]
Let
\[
   R=\{r:1\le r\le \Delta_P,\ |P_r|\ge s\},
   \qquad
   T=\{t:1\le t\le \Delta_Q,\ |Q_t|\ge s\}.
\]
For \(r\in R\) and \(t\in T\), the sets \(P_r\) and \(Q_t\) are disjoint
and have order at least \(s\). Hence \eqref{eq:weighted-layer-count}
and the density hypothesis imply
\[
\begin{aligned}
   \sum_{xy\in E_F(P,Q)}w_xw_y
   &\ge\sum_{r\in R}\sum_{t\in T}e_F(P_r,Q_t)\\
   &\ge\rho q\sum_{r\in R}\sum_{t\in T}|P_r||Q_t|\\
   &=\rho q
      \left(\sum_{r\in R}|P_r|\right)
      \left(\sum_{t\in T}|Q_t|\right).
\end{aligned}
\]
Using the identities for the sets \(P_r\) and the bounds for the sets
\(P_r\) with \(|P_r|<s\), we obtain
\[
   \sum_{r\in R}|P_r|
   =W_P-\sum_{\substack{1\le r\le \Delta_P\\ |P_r|<s}}|P_r|
   \ge W_P-s\Delta_P,
   \qquad
   \sum_{t\in T}|Q_t|\ge W_Q-s\Delta_Q.
\]
Substitution into the preceding inequality proves the lemma.
\end{proof}

The proof of the next probability estimate uses the following bound on
the permanent of a zero-one matrix.

\begin{lemma}[Br\`egman~\cite{Bregman}; see also Schrijver~\cite{Schrijver}]
\label{lem:bregman-minc}
Let $A=(a_{ij})$ be an $m\times m$ zero-one matrix whose $i$th row has
positive sum $r_i$. Let $S_m$ be the set of all permutations of $[m]$,
and define the permanent of $A$ by
\[
   \per(A):=\sum_{\sigma\in S_m}\prod_{i=1}^m a_{i,\sigma(i)}.
\]
Then
\[
   \per(A)\leq \prod_{i=1}^{m}(r_i!)^{1/r_i}.
\]
\end{lemma}

The following lemma applies Lemma~\ref{lem:bregman-minc} to permutations
that avoid a prescribed set of cells. In the proof of
Theorem~\ref{thm:main}, the forbidden cells correspond to matching edges
whose images would have color~1 and hence cannot occur in a color~2
ordered embedding.

\begin{lemma}\label{lem:avoid}
Let $H\subseteq[m]\times[m]$ be a set of forbidden cells, and let
$Z=|H|$. For a uniformly random permutation $\pi$ of $[m]$,
\[
   \bbP\bigl((i,\pi(i))\notin H \text{ for all } i\in[m]\bigr)
      \le \exp\left(-\frac{Z}{2m}\right).
\]
\end{lemma}

\begin{proof}
For each $i\in[m]$, let
\[
   d_i=|\{j\in[m]:(i,j)\in H\}|,
   \qquad
   p_i=m-d_i.
\]
If $p_i=0$ for some $i$, then $(i,\pi(i))\in H$ for every permutation
$\pi$ of $[m]$. In this case the probability in the statement is zero, so
the desired bound holds. We may therefore assume that $p_i\geq1$ for every
$i\in[m]$.

Let $A=(A_{ij})_{i,j\in[m]}$ be the zero-one matrix defined by
\[
   A_{ij}=
   \begin{cases}
      1,&\text{if }(i,j)\notin H,\\
      0,&\text{if }(i,j)\in H.
   \end{cases}
\]
The $i$th row of $A$ has sum $p_i$, and a permutation $\sigma$ of $[m]$
avoids $H$ precisely when $A_{i,\sigma(i)}=1$ for every $i\in[m]$.
Consequently, $\per(A)$ is the number of permutations that avoid $H$.
Since $\pi$ is uniformly random,
\[
   \bbP\bigl((i,\pi(i))\notin H\text{ for every }i\in[m]\bigr)
   =\frac{\per(A)}{m!}.
\]

We use the elementary inequality
\[
   (p!)^{1/p}
   \leq (m!)^{1/m}\left(\frac{p}{m}\right)^{1/2}
\]
for $1\leq p\leq m$.
This follows from the fact that
$f(t)=(t!)^{1/t}/\sqrt{t}$ is nondecreasing in the positive
integer variable $t$.
Indeed, we have
\[
\left(\frac{f(t)}{f(t+1)}\right)^{t(t+1)}
 =
 \prod_{j=0}^{t-1}
 \left(1-\frac jt\right)\left(1+\frac1t\right)^j\leq1,
\]
where the inequality follows from
$1-j/t\leq (1+1/t)^{-j}$ for $0\leq j\leq t-1$.

Apply Lemma~\ref{lem:bregman-minc} to $A$ and then apply this estimate to
each row sum $p_i$. Since $p_i/m=1-d_i/m$ and $1-x\leq e^{-x}$ for
$x\geq0$, we obtain
\[
\begin{aligned}
   \bbP\bigl((i,\pi(i))\notin H\text{ for every }i\in[m]\bigr)
   &=\frac{\per(A)}{m!}
   \leq\frac1{m!}\prod_{i=1}^{m}(p_i!)^{1/p_i}
   \leq\prod_{i=1}^{m}\left(\frac{p_i}{m}\right)^{1/2}\\
   &=\prod_{i=1}^{m}\left(1-\frac{d_i}{m}\right)^{1/2}
   \leq\exp\left(-\frac{1}{2m}\sum_{i=1}^{m}d_i\right)\\
   &=\exp\left(-\frac{Z}{2m}\right),
\end{aligned}
\]
where the final equality uses $\sum_{i=1}^{m}d_i=|H|=Z$.
This proves the lemma.
\end{proof}

\section{Proof of Theorem~\ref{thm:main}}\label{sec:main-proof}

Let $\pi$ be a uniformly random permutation of $[m]$
and define the permutation matching
\[
   M_\pi=\{\{x,m+\pi(x)\}:1\le x\le m\}
\]
on the ordered vertex set $[2m]$. This matching has
interval chromatic number two.
Fix a sufficiently small constant $\delta>0$, to be chosen below, and set
\[
 B:=\delta m^{2/3}(\log m)^{1/3}.
\]
Apply Corollary~\ref{cor:GW} to obtain a triangle-free
graph $F$ on $[B]$.
Write
\[
   q_B=\sqrt{\frac{\log B}{B}},
   \qquad
   s_B=C\sqrt{B\log B},
   \qquad
   L:=\frac{m}{32s_B},
   \qquad
   N:=BL.
\]

Partition the ordered vertex set $[N]$ into
consecutive intervals
\[
   V_1<V_2<\cdots<V_B,
   \qquad |V_s|=L.
\]
Color the edges of the ordered complete graph on
$[N]$ as follows. If
$x\in V_s$, $y\in V_t$, $s<t$, and $st\in E(F)$, then $xy$ receives color $1$.
All remaining edges receive
color $2$. The color-$1$ graph is a blow-up of the triangle-free graph $F$,
and hence contains no triangle.

Let $G_2$ denote the spanning subgraph of the ordered complete graph on
$[N]$ whose edge set consists of all edges of color $2$.
It remains to show that, with high probability, \(G_2\) contains no ordered
copy of \(M_\pi\).

Let $S_m$ denote the set of all permutations of $[m]$, and let
$\bbP_\pi$ denote probability with respect to a uniformly random choice
of $\pi\in S_m$.
The following proposition bounds the probability that $G_2$ contains
the random matching $M_\pi$.

\begin{proposition}\label{prop:randommatching}
There exists an absolute constant $\delta_0>0$ such that, for every fixed
$\delta\in(0,\delta_0]$,
\[
\bbP_{\pi}\bigl(G_2\text{ contains an ordered copy of }M_\pi\bigr)
=o(1)
\]
as $m\to\infty$,
where $\pi$ is chosen uniformly at random from $S_m$.
\end{proposition}

\begin{proof}
Throughout this proof, probabilities are taken with respect to the uniform
random permutation \(\pi\).
Fix $\eps=1/16$ for convenience; no effort is made to optimize
the numerical constants. Since $B=o(m)$, we may assume that $B\leq \eps m/100$.
Suppose that $\phi$ is an order-preserving injection
from $V(M_\pi)$ to $[N]$.
For each $s\in[B]$, define
\[
   I_s=\{x\in[m]:\phi(x)\in V_s\},
   \qquad
   J_s=\{y\in[m]:\phi(m+y)\in V_s\}.
\]
Let $\tau=\max\{s:I_s\ne\emptyset\}$.
Since $\phi$ is order-preserving and $V_1<\cdots<V_B$ are consecutive
intervals, the sets $I_s$ and $J_s$ are possibly empty intervals of size at
most $L$, with $I_s=\emptyset$ for $s>\tau$ and $J_s=\emptyset$ for
$s<\tau$.
Hence, allowing empty intervals, we have
\[
   [m]=I_1\cup\cdots\cup I_\tau,
   \qquad
   [m]=J_\tau\cup\cdots\cup J_B.
\]
In particular, $V_\tau$ is the only block that may contain images from both
sides.

Let $K:=100B/\varepsilon$, and partition $[m]$ into $K$ consecutive
intervals, each of size at most $m/K+1$; call them small intervals.
The next claim provides large subsets of the $I_s$ with $s<\tau$ and
the $J_s$ with $s>\tau$, and bounds the number of possible choices.

\begin{claim}\label{clm:profiles}
For every $\phi$, there are sets $I_s^\circ\subseteq I_s$ for $s<\tau$
and $J_s^\circ\subseteq J_s$ for $s>\tau$, each formed by consecutive
small intervals, such that
\[
\sum_{s<\tau}|I_s^\circ|\ge(1-\varepsilon)m-L
\qquad\text{and}\qquad
\sum_{s>\tau}|J_s^\circ|\ge(1-\varepsilon)m-L.
\]
Let
\[
\mathcal P
=\bigl(I_1^\circ,\ldots,I_{\tau-1}^\circ;
       J_{\tau+1}^\circ,\ldots,J_B^\circ\bigr).
\]
There are at most $\exp(32B)$ possible pairs $(\mathcal P,\tau)$.
\end{claim}

\begin{proof}
For each $s<\tau$, let $I_s^\circ$ be the union of all small intervals
contained in $I_s$. For each $s>\tau$, let $J_s^\circ$ be the union of
all small intervals contained in $J_s$.

Since the nonempty sets $I_s$ form an interval partition of $[m]$
with at most $B$ parts, at most $B-1$ small intervals are not
contained in a single $I_s$.
Therefore, at most
\[
(B-1)\left(\frac{m}{K}+1\right)
\le B\left(\frac{m}{K}+1\right)
=\frac{\varepsilon m}{100}+B
\le\frac{\varepsilon m}{50}
\]
vertices outside $I_\tau$ fail to lie in any $I_s^\circ$ with $s<\tau$.
The same argument applied to $J_\tau,\dots,J_B$ gives the corresponding
bound for the sets $J_s^\circ$ with $s>\tau$.
Since $|I_\tau|,|J_\tau|\le L$, we obtain
\[
\sum_{s<\tau}|I_s^\circ|\ge(1-\varepsilon)m-L
\qquad\text{and}\qquad
\sum_{s>\tau}|J_s^\circ|\ge(1-\varepsilon)m-L.
\]

We next bound the number of possible pairs $(\mathcal P,\tau)$.
Fix $\tau$ and consider the sets $I_s^\circ$.
The $K$ small intervals can be split into groups according to the
$\tau-1$ sets $I_s^\circ$ and the $\tau$ gaps before, between, and after
these sets. The sizes of these $2\tau-1$ groups are nonnegative integers
that sum to $K$ and determine the sets $I_s^\circ$.
We append zeros to obtain a sequence of $2B+1$ nonnegative integers with
the same sum. Hence the number of possible choices is at most
\[
\binom{K+(2B+1)-1}{(2B+1)-1}=\binom{K+2B}{2B}.
\]
The same bound holds for the sets $J_s^\circ$.
Since $K=100B/\eps$ and $\eps=1/16$, the number of pairs
$(\mathcal P,\tau)$ is at most
\[
    B\binom{K+2B}{2B}^{\!2}
      \le B\left(\frac{e(K+2B)}{2B}\right)^{4B}
      =B(801e)^{4B}
      \le \exp(32B)
\]
for all sufficiently large $B$.
\end{proof}

Let $(\mathcal P,\tau)$ range over all pairs obtained from
order-preserving injections $\phi\colon[2m]\to[N]$ by the construction
in the proof of Claim~\ref{clm:profiles}.
This collection does not depend on $\pi$ and contains at most
$\exp(32B)$ pairs.

For each such pair, let $\mathcal E(\mathcal P,\tau)$ be the event that
some injection giving this pair is an ordered embedding of $M_\pi$
into $G_2$. Let $\mathcal B$ be the event that $G_2$ contains an ordered
copy of $M_\pi$. Then $\mathcal B$ occurs if and only if
$\mathcal E(\mathcal P,\tau)$ occurs for at least one of these pairs.
Hence
\[
\mathcal B
=\bigcup_{(\mathcal P,\tau)}\mathcal E(\mathcal P,\tau).
\]
Fix one such pair.
To bound $\bbP(\mathcal E(\mathcal P,\tau))$, define the forbidden board
\[
   H(\mathcal P,\tau)=
      \bigcup_{\substack{s<\tau<t\\ \{s,t\}\in E(F)}}
      \bigl(I_s^\circ\times J_t^\circ\bigr)
      \subseteq [m]\times[m].
\]
Suppose that $\mathcal E(\mathcal P,\tau)$ occurs, and let $\phi$
be an ordered embedding witnessing this event.
If $(x,\pi(x))\in H(\mathcal P,\tau)$ for some $x\in[m]$, then
there exist $s<\tau<t$ with $\{s,t\}\in E(F)$ such that
$x\in I_s^\circ$ and $\pi(x)\in J_t^\circ$.
Thus $\phi(x)\in V_s$ and $\phi(m+\pi(x))\in V_t$, so the
matching edge $\{x,m+\pi(x)\}$ is mapped to an edge of color~1,
a contradiction. Hence $\pi$ avoids $H(\mathcal P,\tau)$, and
\begin{equation}\label{eq-1}
   \bbP(\mathcal E(\mathcal P,\tau))
   \le
   \bbP\bigl((i,\pi(i))\notin H(\mathcal P,\tau)
             \text{ for all }i\in[m]\bigr).
\end{equation}

We next estimate the size of $H(\mathcal P,\tau)$.
For $s<\tau$ put $w_s=|I_s^\circ|$, and for $t>\tau$ put
$w_t=|J_t^\circ|$. Set
\[
   W_P=\sum_{s<\tau}w_s,
   \qquad
   W_Q=\sum_{t>\tau}w_t,
   \qquad
   \Delta_P=\max_{s<\tau}w_s,
   \qquad
   \Delta_Q=\max_{t>\tau}w_t.
\]
Then $0\le w_s,w_t\le L$, so $\Delta_P,\Delta_Q\le L$, and, for all
sufficiently large $m$,
\[
   W_P,W_Q
      \ge (1-\eps)m-L
      \ge \frac34 m
      \ge 8s_B L.
\]
Here we used $\eps=1/16$ and $L=m/(32s_B)$. In particular,
$W_P\ge s_B\Delta_P$ and $W_Q\ge s_B\Delta_Q$.
Let
\[
   P=\{s<\tau:w_s>0\},
   \qquad
   Q=\{t>\tau:w_t>0\}.
\]
Then $P$ and $Q$ are disjoint and nonempty, and the weights
$(w_s)_{s\in P}$ and $(w_t)_{t\in Q}$ are at most $L$.
By Corollary~\ref{cor:GW}, the density condition in
Lemma~\ref{lem:weighted-cut-density} holds with $q=q_B$ and $s=s_B$.
Since the rectangles $I_s^\circ\times J_t^\circ$ are pairwise
disjoint, Lemma~\ref{lem:weighted-cut-density} yields
\[
\begin{aligned}
   |H(\mathcal P,\tau)|
   &=\sum_{\substack{s<\tau<t\\ \{s,t\}\in E(F)}}w_sw_t\\
   &\geq \rho q_B(W_P-s_B\Delta_P)(W_Q-s_B\Delta_Q)\\
   &\geq \rho q_B(W_P-s_BL)(W_Q-s_BL)
    \geq \frac{49\rho}{64}q_BW_PW_Q
     \geq \frac{\rho}{2}q_BW_PW_Q.
\end{aligned}
\]
For each pair $(\mathcal P,\tau)$, this bound and
$W_P,W_Q\ge 3m/4$ imply
$|H(\mathcal P,\tau)|\ge 9\rho q_Bm^2/32$.
Therefore, by \eqref{eq-1}, Lemma~\ref{lem:avoid},
and the union bound,
\[
\begin{aligned}
   \bbP(\mathcal B)
   &\le \sum_{(\mathcal P,\tau)}
      \bbP(\mathcal E(\mathcal P,\tau))
      \le \sum_{(\mathcal P,\tau)}
      \exp\left(-\frac{|H(\mathcal P,\tau)|}{2m}\right)
   \le \exp\left(32B-\frac{9\rho}{64}q_Bm\right),
\end{aligned}
\]
where the last inequality uses Claim~\ref{clm:profiles}
and the preceding bound on $|H(\mathcal P,\tau)|$.
Finally, we have
\[
B=O\!\left(\delta m^{2/3}(\log m)^{1/3}\right)
\quad\text{and}\quad
q_Bm=\Omega\!\left(\delta^{-1/2}m^{2/3}(\log m)^{1/3}\right),
\]
where the constants in the $O$ and $\Omega$ notation are absolute.
Take $\delta>0$ sufficiently small.
Then there is an absolute constant $c_1>0$ such that
\[
   32B-\frac{9\rho}{64}q_Bm
   \leq -c_1m^{2/3}(\log m)^{1/3}
\]
for all sufficiently large $m$. Hence $\bbP(\mathcal B)=o(1)$,
which proves the proposition.
\end{proof}

Choose $\delta>0$ sufficiently small.
Proposition~\ref{prop:randommatching} and the fact that the graph of
color~1 is triangle-free together imply that, asymptotically almost surely,
\[
\rt(K_3,M_\pi)>N=BL\ge c'\frac{m^{4/3}}{(\log m)^{1/3}}
\ge c\frac{n^{4/3}}{(\log n)^{1/3}},
\]
where \(n=2m\) and $c',c>0$ are absolute constants.
This proves Theorem~\ref{thm:main}.

\bigskip\noindent
\textbf{Concluding remarks.}
We have improved the lower bound for ordered Ramsey numbers of random
matchings with interval chromatic number two versus triangles by combining
pseudorandom triangle-free graphs with a coarse encoding of embeddings.
The gap between our lower bound
$\Omega(n^{4/3}/(\log n)^{1/3})$ and the upper bound $O(n^{7/4})$
of Balko and Poljak~\cite{BalkoPoljak} remains open.
A natural next step is to further narrow the gap between the lower and upper bounds.

\bigskip\noindent
\textbf{Declaration on the use of AI.}
The authors used generative AI tools to support discussions of possible
approaches and to improve exposition. All mathematical arguments, results,
and conclusions were verified by the authors, who take full responsibility
for them.

\end{document}